\documentclass[reqno,10pt]{amsart}
\usepackage[english]{babel}
\usepackage{amsthm}
\usepackage{amsmath}
\usepackage{amssymb}
\usepackage[latin1]{inputenc}
\newcommand{\onto}{{\rightarrow}}
\newcommand{\df}{{\rm d}}
\def\R{\text{$\mathbb{R}$}}

\def\fis#1{\dot{#1}}
\def\lra{\longrightarrow}

\def\fis#1{\dot{#1}}
\def\lra{\longrightarrow}

\def\d1#1#2{\frac{d#1}{d#2}}

\def\p1#1#2{\frac{\partial #1}{\partial #2}}
\def\to{t_o}

\def\part{a=\to \leq t_1 \leq \ldots \leq t_{n-1} \leq t_{n}=b}

\def\N{\text{$\mathbb{N}$}}
\def\Z{\text{$\mathbb{Z}$}}

\renewcommand\O{{\rm O}}

\title[$t$-periodic light rays via Finsler geometry]{$t$-periodic light rays in conformally stationary spacetimes via Finsler geometry}
\author[L. Biliotti]{Leonardo Biliotti}
\author[M. A. Javaloyes]{Miguel \'Angel Javaloyes}
\address{Dipartimento di Matematica,
Universit\`a di Parma, Via G. Usberti, 53/A 43100,
Parma Italy}
\email{leonardo.biliotti@unipr.it}
\address{Departamento de Matem\'atica,
Instituto de Matem\'atica e Estat\'\i stica  Universidade de S\~ao
Paulo,  Caixa Postal 66281, CEP 05315--970, SP Brazil}
\thanks{The first author was partially supported by I.N.d.A.M. The authors wish to thank Prof. E. Caponio, P. Piccione, M. S\'anchez and D. Tausk for their valuable comments and their continuous support.}
\email{majava@ime.usp.br}
\begin{document}
\newtheorem{thm}{Theorem}[section]
\newtheorem{prop}[thm]{Proposition}
\newtheorem{lemma}[thm]{Lemma}
\newtheorem{cor}[thm]{Corollary}
\theoremstyle{definition}
\newtheorem{defini}[thm]{Definition}
\newtheorem{notation}[thm]{Notation}
\newtheorem{exe}[thm]{Example}
\newtheorem{conj}[thm]{Conjecture}
\newtheorem{prob}[thm]{Problem}
\theoremstyle{remark}
\newtheorem{rem}[thm]{Remark}
\begin{abstract}
In this paper we prove several multiplicity results of $t$-periodic light rays in conformally stationary spacetimes using the Fermat metric and the extensions of the classical theorems of Gromoll-Meyer and Bangert-Hingston to Finsler manifolds. Moreover, we exhibit some stationary spacetimes with a finite number of $t$-periodic light rays and compute a lower bound for the period of the light rays when the flag curvature of the Fermat metric is $\eta$-pinched.
\end{abstract}
\maketitle
\section{introduction}
The aim of this paper is to show several  multiplicity results of $t$-periodic light rays in conformally stationary spacetimes.
The notion of periodic trajectories in spacetimes makes sense whenever there is a system of coordinates in which the metric coefficients do not depend on the timelike coordinate or depend periodically. The first time that these trajectories appear in literature up to our knowledge is the reference \cite{BenFor90}, where V. Benci and D. Fortunato study timelike $t$-periodic trajectories in static spacetimes. Contemporarily to this work several results on the static case were published (see \cite{BenForGia91,BenForGia91b,greco89}). These results were generalized by considering more general classes of Lorentzian manifolds than static spacetimes; it was studied the stationary case (see \cite{masiello92,sanchezproc99}) and time-dependent orthogonal splittings (see \cite{greco90,masiello95,maspis91}). Furthermore, in \cite{sanchez99}, M. S\'anchez obtains multiplicity results for the static case.

On the other hand, there are few works studying the more specific problem of existence of lightlike $t$-periodic geodesics (light rays). The basic references are the work \cite{candela96} by A. M. Candela  and the reference \cite{MasPic98}, where A. Masiello and P. Piccione use a shortening method to reobtain the results in \cite{candela96} without the constraint of differentiability on the boundary. Moreover, in \cite{sanchez99,sanchezproc99} although it is studied the timelike case, M. S\'anchez makes some remarks about the lightlike one. Following the notation in the above references, we say that a trajectory is $t$-periodic when we do not want to specify the period, otherwise we say that the trajectory is $T$-periodic, being $T$ the period of the trajectory.

The multiplicity results obtained in the timelike case concern always different $T$-periods. For example, in a standard stationary spacetime, that is, a spacetime that can be expressed as $(M\times\R,l)$,
where $l$ is the Lorentzian metric obtained from a Riemannian metric $g_0$ in $M$,  a positive differentiable function $\beta:M\rightarrow \R$ and a vector field in $M$, $\delta\in \mathfrak{X}(M)$ as
\begin{equation}\label{stationarymetric}
l(x,t)[(v,\tau),(v,\tau)]=g_0(x)[v,v]+2g_0(x)[\delta(x),v]\tau-\beta(x)\tau^2,
\end{equation}
(here $(x,t)\in M\times\R$ and $(v,\tau)\in T_xM\times\R$),
 they show (see \cite[Theorem 1.1]{sanchezproc99}) that whenever $\pi_1 (M)$ is not trivial and ${\mathcal C}$ is a non-trivial class of conjugacy of $\pi_1 (M)$, there exists $T_{\mathcal C}\in\R$, such that
\begin{itemize} 
 \item for every $T>T_{\mathcal C}$ there exists a $T$-periodic timelike geodesic $(x,t)$ of the spacetime $(M\times\R,l)$ such that $x\in {\mathcal C}$,
 \item there exists a $T_{\mathcal C}$-periodic light ray $(x,t)$ with $x\in{\mathcal C}$ 
 \item and for $0<T\leq T_{\mathcal C}$ there is no $T$-periodic timelike geodesic with $x$ in ${\mathcal C}$.
 \end{itemize}  
 We observe that two $t$-periodic trajectories are non-equivalent or geometrically distinct when the images do not coincide and there is no time traslation that brings one into the other. Even if the fundamental group of $M$ is infinite, the result in \cite{sanchezproc99} does not guarantee multiplicity of non-equivalent $t$-periodic light rays or $T$-periodic timelike geodesics; in fact, the trajectories corresponding to different conjugacy classes may be iterations of a unique trajectory. A similar problem occurs with the multiplicity results in \cite{candela96}. The goal of this paper is to use the relation between lightlike geodesics in a standard stationary spacetime and closed geodesics in the base $M$ of a certain Finsler metric introduced in \cite{cmm} called ``the Fermat metric''. Then a lightlike geodesic is $t$-periodic if and only if the projection in $M$ is a closed pregeodesic for the Fermat metric. Furthermore, as we comment in Subsection~\ref{conformal}, the results are still valid in the more general class of  spacetimes admitting a complete conformal Killing vector field with section.

 In the static case the problem is simpler since the Fermat metric is in fact Riemannian. We observe that in this situation the Fermat metric coincides with the Jacobi metric introduced by M. S\'anchez in \cite{sanchez99} for $E=0$ and as the author observes (see the comments after Corollary 3.10 in \cite{sanchez99}), multiplicity results for Riemannian closed geodesics can be used.  The main results on the existence of infinitely many geometrically distinct closed geodesics are the celebrated paper of D. Gromoll and W. Meyer \cite{GroMey69b} and the more recent result of V. Bangert and N. Hingston \cite{bh}. In the same paper (see Corollary 3.10 in \cite{sanchez99}) the author shows the existence of at least two distinct $t$-periodic light rays in static spacetimes. This result cannot be reproduced in the stationary context unless you show the existence of two closed Finslerian geodesics. The reason for that, it is the use of  the reversibility of Riemannian geodesics for the proof in the static case (each orientation gives a different $t$-periodic light ray), whereas the Fermat metric is non-reversible. Some recent results show the existence of two closed Finsler geodesics in some particular cases (see \cite{bl,HuLong07,rade06}).

 Bearing in mind the Fermat metric, we exploit multiplicity results of Finsler closed geodesics; a  Gromoll-Meyer type theorem was obtained in the context of Finsler metrics by H. Matthias (see \cite{matthias78,matthias80}), whereas we include in this work an extension of the Bangert-Hingston theorem. In this way we get two theorems of existence of infinitely many  $t$-periodic light rays: Theorem \ref{gmstationary} when the Cauchy surface $M$ satisfies the Gromoll-Meyer condition and Theorem \ref{bhstationary} when the fundamental group of $M$ is infinite abelian. None of the theorems apply when the Cauchy surface is $S^3$, so that it is a natural question if there do exist infinitely many $t$-periodic light rays in this situation. In Section \ref{katok} we provide several examples of spacetimes with only a finite number of $t$-periodic light rays, that we obtain from the classical Katok metrics. Summing up, we prove that there exist infinite $t$-periodic light rays when $\limsup b_k(\Lambda M)=+\infty$ or $\pi_1(M)$ is infinite abelian and we exhibit some examples with a finite number of $t$-periodic light rays with $M$ being $S^{2n}$, $S^{2n-1}$, $P^nC$, $P^nH$ or $P^2Ca$.  Another interesting problem is if the period $T$ of the light ray can be arbitrarily small. As $T$ coincides with the length of the Fermat geodesics associated to the light ray, the answer is no, because the inyectivity radius of a compact Finsler manifold has a lower bound. Moreover, we obtain some estimates of this bound imposing some restrictions on the flag curvature. In particular, when the flag curvature is $\eta$-pinched, we can establish a lower bound for the period $T$ using the reversibility of the Fermat metric that we compute in Subsection~\ref{reversibility}.
 \section{Multiplicity results for closed geodesics on Finsler manifolds.}
\subsection{Conformally stationary spacetimes}\label{conformal} We begin by observing that the study of lightlike geodesics in conformally standard stationary spacetimes can be reduced to a standard stationary ambient, because the support of lightlike geodesics is preserved by conformal changes (see for example \cite{Piccio97}). In the following we give sufficient conditions for a spacetime to be conformally standard stationary.
Let $(\bar{M},\bar{g})$ be a  spacetime endowed with a timelike complete conformal Killing vector field $K$ admitting a section, that is, a timelike complete vector field satisfying ${\mathcal L}_K{\bar{g}}=\lambda \bar{g}$ for some  function $\lambda:\bar{M}\rightarrow\R$ and such that there exists a spacelike hypersurface $S$ meeting exactly at one instant every integral line of $K$. Then, $(\bar{M},\bar{g})$ is globally conformal to a standard stationary spacetime. To see this, it is enough to consider the conformal Lorentzian metric $g=-\frac{1}{\bar{g}(K,K)}\bar{g}$. For this metric, $K$ is a Killing field satisfying $g(K,K)=-1$ (see \cite[Lemma 2.1]{sanchez97}). Moreover as we have assumed that $K$ is complete and it admits a section, $(\bar{M},g)$ is standard stationary. In \cite{JS08}, M. S\'anchez and the second author give a more precise result.
\begin{lemma}
A distinguising spacetime $(\bar{M},\bar{g})$ endowed with a timelike complete conformal Killing field $K$  is globally conformal to a standard stationary spacetime.
\end{lemma}
\begin{proof}
See \cite[Theorem 1.2]{JS08}.
\end{proof}
From now on, we will assume that $M\times\R$ is endowed with the metric \eqref{stationarymetric} and will state all the theorems for standard stationary spacetimes, bearing in mind that the same results hold in the more general class described above.
For more information about Lorentzian geometry we suggest the references \cite{BeErEa96,HaEl73,One83} and for the physical interest of stationary spacetimes, \cite[Chap. 7]{SaWu77}.
\subsection{Standard stationary spacetimes and the Fermat metric}
As it was observed in \cite{cmm}, Fermat's principle implies that a curve $(x,t):[0,1]\onto M\times\R$ is a future-pointing lightlike geodesic
if and only if the projection $x:[0,1]\onto M$ is a pregeodesic for the Fermat metric on $M$, that is, the Finsler metric given by
\begin{equation}\label{fermatmetric}
F(x,v)=\frac{1}{\beta(x)}g_0(\delta(x),v)+\frac{1}{\beta(x)}\sqrt{g_0(\delta(x),v)^2+\beta(x) g_0(v,v)},
\end{equation}
where $(x,v)\in TM$, and the time component is
\begin{equation*}
t(s)=t_0+\int_0^s\left( \frac{1}{\beta(x)}g_0(\delta(x),v)+\frac{1}{\beta(x)}\sqrt{g_0(\delta(x),v)^2+\beta(x) g_0(v,v)}\right)\df s,
\end{equation*}
that is, the projection of a lightlike geodesic is a Fermat geodesic up to a re\-pa\-ra\-me\-tri\-za\-tion and
  $t(s)-t_0$ is the Fermat length of $x:[0,s]\onto M$.
 We say that a lightlike geodesic $\gamma=(x,t):[0,1]\onto M\times\R$ is $t$-periodic when there exists a positive real $T>0$ such that $x$ is a smooth closed curve in $M$, $t(1)=t(0)+T$ and $\dot t(1)=\dot t(0)$. If we want to specify the period, we say that the light ray is $T$-periodic. Furthermore, we say that two $t$-periodic trajectories are equivalent if the images coincide up to a constant traslation in the time component. We are interested in multiplicity results of non-equivalent $t$-periodic lightlike geodesics. In the setup of Fermat metrics or more generally, Finsler manifolds, this is equivalent to show multiplicity of geometrically distinct closed Finsler geodesics, that is, closed geodesics with distinct support.

\subsection{Finsler metrics} We recall some basic facts about Finsler manifolds and we refer
the reader to \cite{bcs} for notations and for any further
information.
Let $M$ be a smooth, real, paracompact, connected manifold of finite dimension.
A Finsler structure on $M$ is a continuous non negative function
$F:TM \rightarrow \R$ which is  smooth  on $TM \setminus 0$,
vanishing only on the zero section, fiberwise positively
homogeneous of degree one, i.e. $F(x,\lambda y)=\lambda F(x,y)$,
for all $x\in M$, $y\in T_x M$ and $\lambda>0$ and which has fiberwise strongly convex square, that is,
\[g_{ij}(x,y)=\left[\frac 12\frac{\partial^2 (F^2)}{\partial y^i\partial y^j}(x,y)\right]\]
is positively defined for any $(x,y)\in TM\setminus 0$. We will also use the tensor 
\begin{equation}\label{fundtensor}
g=g_{ij}(x,y)\df x^i\otimes\df x^j,
\end{equation}
that is usually called {\it the fundamental tensor} of $(M,F)$ and it is a symmetric section of the tensor bundle $\pi^*(T^*M)\otimes\pi^*(T^*M)$, where $\pi^*(T^*M)$ is the pulled-back tensor of the projection $\pi^*:T^*M\rightarrow M$ through the map $\pi:TM\setminus 0\rightarrow M$
(see \cite{bcs}). The length of a
piecewise smooth curve $\gamma:[a,b]\lra M$ with respect to the
Finsler structure $F$ is defined by
\[
L(\gamma)= \int_{a}^{b} F(\gamma(s), \fis{\gamma}(s))ds,
\]
while the Energy of $\gamma$ as
\[
E(\gamma)=\int_{a}^{b} F^2(\gamma(s),\fis{\gamma}(s)) ds.
\]
Using the H\"older's inequality,  it is easy to check that $L(\gamma) \leq E(\gamma)^{\frac{1}{2}}
|b-a|^{\frac{1}{2}}$.

In this section we always assume that $M$ is a compact manifold and we fix a
Riemannian metric $h$. In particular
there exist two positive constants $c_1,c_2$ such that
\begin{equation} \label{tu}
c_1 h(v,v) \leq F^2(x,v) \leq c_2 h(v,v), \forall (x,v)\in TM.
\end{equation}
Let $S^1$ be the unit circle, viewed as the quotient $[0,1] /
\{0,1\}$ and denote by $\Lambda M$ (or simply $\Lambda$) the infinite dimensional Hilbert
manifold of all loops $\gamma:S^1 \lra M$, i.e.
$\gamma(0)=\gamma(1)$, of Sobolev class $H^1$; with $\Omega(M,p)$ we
denote the loop space of $M$, i.e. all loops $\gamma\in \Lambda M$
such that $\gamma(0)=\gamma(1)=p$ for some $p\in M$. As all the manifolds $\Omega(M,p)$ with $p\in M$ are homotopically equivalent and we only need to consider the homotopy groups of $\Omega(M,p)$, we will use simply $\Omega$. The infinite dimensional manifold $\Lambda$ will be endowed with the Hilbert space inner product:
\begin{equation}\label{product}
\langle V, W\rangle =\int^{1}_{0}[ h(V,W) + h( V', W') ]dt,
\end{equation}
where $V$ and $W$ are $H^1$-vector fields along an $H^1$-curve $\gamma$ and $'$ denotes the covariant differentiation along $\gamma$ associated to
the Levi-Civita connection of the metric $h$.

 Note that the
length and the energy are defined on $\Lambda M$ and their critical
points are exactly the closed geodesics, affinely parametrized when considering the energy. Moreover the energy
functional defined on $\Lambda$ is $C^{2-}$, i.e. $C^1$ and the
differential is locally Lipschitzian and it satisfies the
Palais-Smale condition (see \cite{me} and also \cite{cmm}).
The compact group $S^1$ acts equivariantly on $\Lambda M$ via the operation of $S^1$
on the parameter circle $S^1$ and the orbits of this action are smooth compact submanifolds.
In particular the critical points of $E$ are never isolated. We remark, that in the Riemannian case $S^1$ must be substituted by $\O(2)$ because of the reversibility.
\subsection{Gromoll-Meyer theorem}
The most celebrated result in the theory of closed geodesics and its multiplicity is that of D. Gromoll and W. Meyer (see \cite{GroMey69b}). For many years it was looked for a topological condition assuring the existence of infinitely many geometrically  distinct closed geodesics in a compact manifold. The bigger difficulty to reach this goal is that Lyusternik-Schnirelmann theory does not distinguish when a closed geodesic is prime or not, that is, when it is the iteration of another closed geodesic. On the other hand, due to the fact that there is an $\O(2)$-action  that preserves closed geodesics,  we must use equivariant Morse theory in order to distinguish the tower generated by the $\O(2)$-action, that is, closed geodesics obtained from a fixed one by applying rotations. In order to show the multiplicity of geometrically distinct closed geodesics Gromoll and Meyer imposed the following condition on the topology of the free loop space of the manifold $M$, that we denote $\Lambda M$. Let $b_k(\Lambda M)$ be the $k$-th Betti number of $\Lambda M$. Then $M$ satisfies the Gromoll-Meyer condition  if it is simply connected and $\sup b_k(\Lambda M)=+\infty$. We observe that  the Gromoll-Meyer condition is empty in dimension $3$ if we accept the Poincar\'e conjecture. This is the most interesting case in General Relativity, but in \cite[Remark 8.4]{bimerpi07} it is said that the same proof works when assuming the condition $\limsup_{k\rightarrow \infty} b_k(\Lambda M)=+\infty$, that is in fact non-empty in dimension $3$. In the generalized version proved by H. Matthias for Finsler metrics in \cite{matthias78,matthias80} the theorem states the following.
\begin{thm}\label{gromollmeyer}
Let $(M,F)$ be a compact Finsler manifold satisfying the Gromoll-Meyer condition. Then there exist infinitely many geometrically distinct closed geodesics.
\end{thm}
By applying Theorem \ref{gromollmeyer} and Fermat's principle to the Fermat metric we obtain the following result about multiplicity of $t$-periodic light rays.
\begin{thm}\label{gmstationary}
Let $(M\times\R,l)$ be a standard stationary spacetime with the metric $l$ as in \eqref{stationarymetric}. If $M$ is compact and satisfies the Gromoll-Meyer condition, then there exist infinitely many non-equivalent $t$-periodic light rays in $(M\times\R,l)$.
\end{thm}
\subsection{Bangert-Hingston theorem.} The second result on the multiplicity of closed geodesics on Riemannian metrics is that of V. Bangert and N. Hingston in \cite{bh}. They prove that when the fundamental group of $M$ is infinite abelian there is always an infinite number of geometrically distinct closed geodesics.  We observe that there is another result of W. Ballmann (cf. \cite{bal86}) on the multiplicity of closed geodesics on Riemannian manifolds, assuming that the fundamental group of $M$ is almost nilpotent and contains a copy of $\Z$. It would be interesting to generalize this result to Finsler manifolds. We will include a proof of the Bangert-Hingston theorem in the Finsler case even if it is very similar to that of the Riemannian case in order to clarify some points. We observe that we have slightly modified the proof in \cite{bh} to take in consideration the energy functional rather than the length one, which allows us to apply Lyusternik-Schnirelmann theory to the space of $H^1$-curves.
\begin{thm}\label{bangerthingston}
Let $(M,F)$ be a compact Finsler manifold of dimension $\geq2$
whose fundamental group is infinite abelian. Then there exist infinitely many geometrically
distinct non trivial closed geodesics in $M$.
\end{thm}
\begin{proof}
Firstly we assume that $\pi_1(M)=\Z$.

Given $\gamma \in \Lambda$, we may define the \emph{m}-th iterate of
$\gamma$, which will be denoted by $\gamma^m$, as
$\gamma^m(s)=\gamma(ms)$. Let $t$ be a generator of $\pi_1(M)$. We
will denote by $\Lambda_m$ the connected component of $\Lambda$
which contains $t^m$. Since $M$ is not homotopy equivalent to a
circle then $\pi_n(M)\neq 0$ for some (minimal) $n>1$.  Indeed, if
$f:S^1 \lra M$ is a loop such that $[f]=t$, then $f_{\#}:\pi_k(S^1
) \lra \pi_k (M)$ is an isomorphism for every $k\in \N$, hence by Whitehead's theorem (see \cite{wh}), $M$ is homotopy equivalent to $S^1$. But this is impossible, because $H_l(M,\Z_2)=\Z_2\not= H_l(S^1,\Z_2)=0$, being $l=\dim (M)$.

 Lemma 1 and 2 in \cite{bh} provide the existence of
a $k\in \N$ such that,  for all $m\in \N$ there exist
non-trivial classes $\alpha_m\in \pi_{n-1}(\Lambda_{mk},
\gamma_m)$ which are in the image of $\pi_{n-1}(\Omega,\gamma_m)$.
Let
$\gamma_m \in \Lambda_{mk}$ be a closed geodesic with length
$
k_m = \inf \{ L(\gamma): \gamma \in \Lambda_{mk} \}.
$
Define
\[
\tau^2_m= \inf_{f\in \alpha_m } \sup \{E(\gamma):\gamma \in
\mathrm{Im}f \}.
\]
If we consider minus the gradient of the functional $E$ on $\Lambda$ with respect to the
Hilbert inner product $\big \langle \cdot, \cdot \big\rangle$ in \eqref{product}, then its flow leaves
$\alpha_m \in \pi_{n-1}(\Lambda_{mk}, \gamma_m)$ invariant since $\gamma_m$ is a critical point.
Hence, we can apply Lyusternik-Schnirelmann theory as in \cite[9.2.5 Minimax Principle]{PalaisTerng89}, and therefore  there exists a closed geodesic
$\delta_m \in \Lambda_{m}$ with length $\tau_m$.  Moreover, using again Lemma 2 in \cite{bh}, we
deduce the inequality
\begin{equation} \label{io}
k_m \leq \tau_m \leq\sqrt{\frac{1}{2}( k^2_m +P^2)}\leq k_m+P,
\end{equation}
for a constant $P>0$ independent of $m$.

Let us first assume that
$\tau_m= k_m$ for a certain $m\in\N$ and  there is only a finite number
of geometrically distinct closed geodesics. In this context we may
prove the following result.
\begin{lemma}\label{lemma} There exists $f\in
\alpha_m$ such that $\sup \{E(\gamma): \gamma \in {\rm Im} f \}=\tau^2_m$ and such that $f(S^{n-1})$ is contained in the $S^1$-orbit of a closed geodesic in $\Lambda_{mk}$.
\end{lemma}
\begin{proof}
Let $\beta_1, \ldots, \beta_p $ be a maximal set of pairwise geometrically
distinct closed geodesics in $M$ which belong to $\Lambda_{mk}$. Let $
S^1 \cdot \beta_i$ be the orbit of $S^1$ through $\beta_i$
and let $\nu(S^1 \cdot \beta_i)$ be the normal bundle of $S^1\cdot \beta_i $ in $\Lambda$.
Then
there exists $r>0$ such that
\[
\mathcal{N}_i=\{ v\in \nu(S^1 \cdot \beta_i): \big \langle v,v   \big\rangle <r^2 \}
\]
is a normal disc bundle over $S^1 \cdot \beta_i$
and $\mathcal{N}_1, \cdots, \mathcal{N}_p$  are pairwise disjoint.
We shall prove that there exists $\epsilon>0$ such that
if $(E(\gamma))^{\frac{1}{2}} < \tau_m + \epsilon$, $\gamma \in \Lambda_{mk}$,
then $\gamma \in \mathcal{N}_i$ for some $i \in \{1,\ldots,p\}$.
Otherwise there exists a sequence
$(\phi_j )_{j\in \N}$ in $\Lambda_{mk}$  such that
$$\tau_m=k_m\leq E(\phi_j)^{\frac{1}{2}} <\tau_m +\frac{1}{j}$$ and $\phi_j$ does not
lie in any $\mathcal{N}_i$. From (\ref{tu}) we also have that there exists $K>0$ such that $L^h(\phi_j)\leq K$, where 
$L^h$ is the length functional with respect to the Riemannian metric $h$.
Then, by Ascoli-Arzel\'a theorem, there exists
a subsequence of $(\phi_j )_{j\in \N}$ which converges uniformly
to a continuous curve $\phi$. We will prove that
$\phi \in \Lambda_{mk}$ and $L(\phi)=\tau_m$, i.e.
$\phi$ is a closed geodesic, and there is a subsequence of $\phi_j$ converging strongly to $\phi$, which is an absurd since $\phi_j$ does not lie in $\mathcal{N}_i$ for any $i=1, \ldots, p$. The fact that $\phi \in \Lambda_{mk}$ follows from the uniform convergence (see \cite[Lemma 1.4.7]{jost}). To see that $L(\phi)=\tau_m$ we proceed as follows. Assume that $\phi_j$ is the subsequence uniformly convergent. We consider a partition $t_0=0<t_1<\ldots<t_q<t_{q+1}=1$ in such a way that $\phi_j([t_i,t_{i+1}])$, for $i=0,\dots,q$ is contained in a coordinate neighborhood $(U_i,\varphi_i)$, so that componing with the chart $\varphi_i$ we can think $\phi_j$ as a function with image in $\R^l$ with $l$ the dimension of $M$ (we omit in the following the composition with the chart, that is, we will write $\phi_j$ and $\dot\phi_j$ rather than $\varphi_i\circ\phi_j$ and $(\varphi_i\circ\phi_j)'$, and $F$ and $g$ rather than the composition with the inverse of the natural chart in $TM$, $(\varphi_i,\bar{ \varphi}_i)^{-1}:\varphi_i(U_i)\times\R^l\rightarrow TM$). As $E^h(\phi_j)$ is bounded ($E^h$ the energy functional for the Riemmanian metric $h$) and $\phi_j$ converges uniformly in $[t_i,t_{i+1}]$, the sequence $\phi_j$ admits a weakly convergent subsequent. Since the function $\varphi_{s_0}^u:H^1([a,b];\R^l)\rightarrow \R$  such that $\varphi_{s_0}^u(f)=f_u(s_0)$ is linear and continuous (here $f=(f_1,\ldots,f_l)$), the curve $\phi$ coincides with the weak limit of $\phi_j$, and $\phi$ is an $H^1$-function. 
Set
\[G(\phi_j(s))=\left\{\begin{array}{cl}
g(\phi(s),\dot \phi(s))[\dot\phi(s),\dot\phi_j(s)]&\text{if } \dot\phi_j(s)\not=0,\\
0&\text{if } \dot\phi_j(s)=0,
\end{array}\right.\]
where $g(\phi(s),\dot \phi(s))$ is the fundamental tensor in \eqref{fundtensor}. Applying Cauchy-Schwarz inequality (for Minkowskian norms) and the H\"older's one, we obtain
\begin{multline}\label{desigual}\int_{t_i}^{t_{i+1}} G(\phi_j)\df s\leq \int_{t_i}^{t_{i+1}} F(\phi,\dot\phi)F(\phi,\dot\phi_j)\df s\\\leq \left(\int_{t_i}^{t_{i+1}} F^2(\phi,\dot\phi)\df s\right)^{\frac{1}{2}}\left(\int_{t_i}^{t_{i+1}} F^2(\phi,\dot\phi_j)\df s\right)^{\frac{1}{2}}.\end{multline}
Moreover,  if we name
\[
U_j(s)=\left\{\begin{array}{cl}
\frac{\phi_j(s)}{F(\phi_j(s),\dot\phi_j(s))}&\text{if } F(\phi_j(s),\dot\phi_j(s))\not=0,\\
0&\text{if } F(\phi_j(s),\dot\phi_j(s))=0,
\end{array}\right.\]
then
\begin{equation}\label{jogodedentro}
\int_{t_i}^{t_{i+1}} \left(F^2(\phi,\dot\phi_j)-F^2(\phi_j,\dot\phi_j)\right)\df s=\int_{t_i}^{t_{i+1}} F^2(\phi_j,\dot\phi_j)\left(F^2(\phi,U_j)-F^2(\phi_j,U_j)\right)\df s.\end{equation}
As the image of $\{(\phi,U_j),(\phi_j,U_j)\}_{j\in\N}$ is contained  in a compact subset where $F$ is uniformly continuous, $\int_{t_i}^{t_{i+1}} F^2(\phi_j,\dot\phi_j)\df s$ is bounded and $\phi_j$ converges uniformly to $\phi$, we conclude that the quantity in \eqref{jogodedentro} goes to zero. By the election of $\phi_j$,  we have that
\begin{equation}\label{limitetau}\lim_{j\rightarrow \infty}\sum_{i=1}^q \int_{t_i}^{t_{i+1}} F^2(\phi_j,\dot\phi_j)\df s=\tau^2_m. \end{equation}
From Eq. \eqref{tu} and \eqref{desigual} follows that $\int_{t_i}^{t_{i+1}} G(\phi_j)\df s$ is a continuous operator and as $\phi_j$ converges weakly to $\phi$, 
$$\lim_{j\rightarrow\infty}\int_{t_i}^{t_{i+1}} G(\phi_j)\df s=\int_{t_i}^{t_{i+1}} F^2(\phi,\dot\phi)\df s,$$ 
(here we have used that by Euler's theorem $g(\phi,\dot \phi)[\dot\phi,\dot\phi]=F^2(\phi,\dot\phi)$). Moreover, using the above equation, the square of the inequality \eqref{desigual}, the Eq. \eqref{jogodedentro}, \eqref{limitetau} and the inequality $L(\phi)^2\leq E(\phi)$, we conclude that
\begin{equation*}L(\phi)^2\leq\sum_{i=1}^q\int_{t_i}^{t_{i+1}} F^2(\phi,\dot\phi)\df s\leq \lim_{j\rightarrow\infty}\sum_{i=1}^q\int_{t_i}^{t_{i+1}} F^2(\phi_j,\dot\phi_j)\df s=\tau_m^2.\end{equation*}
As $\tau_m$ coincides with the infimum of the length in $\Lambda_{km}$, we have in fact the equality in the last inequality, so that $\phi$ must be a geodesic and $L(\phi)=\tau_m$. Furthermore, 
from Eq. \eqref{jogodedentro} and \eqref{limitetau} we obtain that
\[\lim_{j\rightarrow\infty}\sum_{i=1}^q\int_{t_i}^{t_{i+1}} F^2(\phi,\dot\phi_j)\df s=\tau_m^2,\]
 and using this together with the fact that $\phi_j$ converges weakly to $\phi$ and the Cauchy-Schwarz inequality for Minkowskian norms we get
\begin{equation*}\lim_{j\rightarrow \infty}\sum_{i=1}^q\int_{t_i}^{t_{i+1}}g(\phi,\dot\phi)[\dot\phi_j-\dot\phi,\dot\phi_j-\dot\phi]\df s=0
,\end{equation*} 
which implies the strong convergence of $\phi_j$.

Choose $f\in \alpha_m$ such that $E(\gamma)^{\frac{1}{2}}< \tau + \epsilon$,
for any $\gamma \in {\rm Im}f$. Hence $f(S^{n-1})$ is contained
in $\mathcal{N}_{i}$  for some $1\leq i \leq p$.
Using the standard deformation retract on the zero section of $\mathcal{N}_i$ we obtain a new
$f$ such that $f(S^{n-1})$ lies in $S^1\cdot \beta_i$ concluding the proof.
\end{proof}
By the above lemma, we know that there exists a representative $f$ of $\alpha_m$ such that $f(S^{n-1})$ lies in an orbit of the
natural $S^1$-action on $\Lambda$. This implies that if $n>2$ then
$\alpha_m=0$, because there is a representative of $\alpha_m$ with the image contained in an $S^1$. This is a contradiction because $\alpha_m$ is non trivial. Assume that $n=2$, then $e_* (\alpha_m) \neq 0$, where
$e:\Lambda \lra M$ is the evaluation map at the base point. However $\alpha_m$
lies in the image of $\pi_1(\Omega,\gamma_m)$, so that it follows that
$e_* (\alpha_m)=0$, which is again a contradiction.

Now we can assume that $\tau_m > k_m$ for every $m\in \N$ and there exits $\epsilon
>0$ such that there is no closed geodesic with length $(k_1,k_1 +
\epsilon]$ in $\Lambda_k$: otherwise there are infinitely many
geometrically distinct closed geodesics in $M$, concluding our
proof.

Let $p$ be a prime number such that $p\,\epsilon >P$ (see \eqref{io}). The
multiplicity of a loop $\gamma \in \Lambda$, $\gamma$ not
homotopically equivalent to a constant loop, is the largest integer
$j$ such that $\gamma=\overline{\gamma}^j$. Note that if $j>k$ and
$\gamma \in \Lambda_{kp}$ then $j=sp$ for some $s\in\N$. We shall prove that
$\gamma_p$ and $\delta_p$ cannot both have multiplicity bigger
than $k$ from which one may deduce that if $p(i)$ is the $i$-th prime,
at least $(i-i_0)/k$ of the geodesics among $\gamma_p$ and
$\delta_p$ with $p\leq p(i)$ are geometrically distinct, where $i_0$ is the cardinality in the sequence of prime numbers of the biggest prime $p$ satisfying $p\,\epsilon\leq P$. Then our
result is proved.

Let us assume that the multiplicity of $\gamma_p$ and $\delta_p$
is bigger than $k$. Hence there exist $\overline{\gamma},
\overline{\delta} \in \Lambda_k$ such that
$\gamma_p=\overline{\gamma}^p$ and $\delta_p=\overline{\delta}^p$.
Since $k_p \leq p\, k_1$ and $L(\gamma_p)=p\, L(\overline{\gamma})$ we
get that $L(\overline\gamma)=k_1$ and $k_p=p\, k_1$. Moreover from $ \tau_p
=L(\delta_p)=p\, L(\overline\delta ) >k_p = p\, k_1$ we have, since
$\overline\delta \in \Lambda_k$, that $L(\overline \delta) > k_1 +
\epsilon$. Hence
\[
\tau_p > k_p + p\,\epsilon >k_p + P,
\]
contradicting $(\ref{io})$.

Assume that $\pi_1(M)$ is infinite abelian and $\pi_1(M) \neq \Z$. Let $t\in \pi_1(M)$
be of infinite order and let $s\in  \pi_1(M)$ such that $s,t$ are independent. We
denote by $\Lambda_m$ the $t^m$ component of $\Lambda$. Denote $k_m=\inf \{ L(\delta): \delta \in \Lambda_m \}$. As in \cite{bh} we may prove
that there exist  homotopy classes $\alpha_m \in \pi_1 (\Lambda_m)$ such that
\begin{itemize}
\item $e_* (\alpha_m)=s$;
\item $\tau^2_m= \inf_{f\in \alpha_m}\{ \sup E(\delta) : \delta \in {\rm Im}f \}$ satisfies
$\tau_m \leq m\, k_1 +P$, where $P$ is a positive constant independent of $m$.
\end{itemize}
If $\tau_m=k_m$ for a certain $m\in\N$, as in Lemma \ref{lemma} we can show that if there is only a finite number of distinct closed geodesics, then $\alpha_m$ admits a representative that lies in an $S^1$ orbit of some geodesic of which its homotopy class lies in the
infinite abelian subgroup of $\pi_1(M)$ generated from $t$
which is an absurd since $e_*(\alpha_m)=s$.

Assume now $\tau_m > k_m$ for every $m\in\N$. Let $\gamma_p$
be a closed geodesic in $\Lambda_p$ with length $k_p$ and let
$\delta_p$ be a closed geodesic such that $L(\delta_p)=\tau_p$. As before,
we may assume that there do not exist closed geodesics in $\Lambda_1$
with length $(k_1,k_1+\epsilon]$.

Let $p$ be a prime number such that $p\,\epsilon>P$.
If the  multiplicity of both $\gamma_p$ and $\delta_p$ is bigger than $1$
then $\gamma_p=\overline{\gamma}^p$ and $\delta_p=\overline{\delta}^p$.
In particular since $k_p \leq p\, k_1$ and $L(\gamma_p)= p\, L(\overline\gamma ) \leq p\, k_1$ we have that
$L(\overline\gamma )=k_1$, $k_p=p\, k_1$ and $L(\overline\delta )> k_1$. Therefore
$L(\overline\delta )> k_1 + \epsilon$ and
\[
\tau_p=L(\delta_p)=p\, L(\bar \delta)>p(\,k_1+\epsilon)> p\, k_1 + P,
\]
which is a contradiction. Then, as before, there are many infinite geometrically distinct closed geodesics in $M$.
\end{proof}

Finally we state the result about the multiplicity of $t$-periodic light rays that follows from Theorem \ref{bangerthingston} and Fermat's principle.
\begin{thm}\label{bhstationary}
Let $(M\times\R,l)$ be a standard stationary spacetime with the metric $l$ as in \eqref{stationarymetric}. If $M$ is compact and its fundamental group is infinite abelian, then there exist infinitely many non-equivalent $t$-periodic light rays in $(M\times\R,l)$.
\end{thm}
\section{Spacetimes with a finite number of $t$-periodic light rays}\label{katok}
\subsection{Fermat metrics, Randers metrics and  Zermelo metrics}\label{subsection:FeRaZer}
In the following, we want to show that the family of Fermat metrics coincides with another two families of Finsler metrics and we will try to get the most information of this fact. We recall that a Fermat metric is given by Eq. \eqref{fermatmetric}. Randers metrics are the most classical example of non-reversible Finsler metrics and they were introduced by G. Randers aiming to study electromagnetics trajectories in spacetimes (see \cite{Rander41}).  These metrics are determined by a Riemannian metric $h$ and a $1$-form $\omega$ in a manifold $M$  as
\begin{equation}\label{randers}
R(x,v)=\sqrt{h(v,v)}+\omega(x)[v],
\end{equation}
where $x\in M$ and $v\in T_xM$ and
such that $|\omega|_x<1$ in every $x\in M$ ($|\cdot|_x$ computed using $h$). The last condition in $\omega$ implies the positivity and fiberwise strongly convex square of the Randers metric $R$ (see \cite{bcs}).
On the other hand, Zermelo metrics were introduced in $\R^2$ by Zermelo (see \cite{Ze31}) to study a problem of navigation and they were generalized by Z. Shen in \cite{Sh03} as follows: given a Riemannian manifold
$(M,g)$ and a vector
field $W$ in $M$ such that $g(W,W)<1$, the Zermelo metric $Z:TM\rightarrow \R$ is given by
\begin{equation}\label{zermelo}
Z(x,v)=\sqrt{\frac{1}{\alpha^2}g(W,v)^2+\frac{1}{\alpha}g(v,v)}-\frac{1}{\alpha}g(W,v),
\end{equation}
where $x\in M$, $v\in T_xM$ and $\alpha=1-g(W,W)$ (along this section we omit the point of evaluation in $M$ to avoid mess). 
\begin{prop}\label{prop:RZF}
Randers, Zermelo and Fermat metrics provide the same family of Finsler metrics.
\end{prop}
\begin{proof}
The fact that Zermelo metrics are Randers comes easily. As it was observed in \cite{BaRoSh04}, a Randers metric can be expressed as a Zermelo metric by taking
\begin{equation*}
g(v,w)=\varepsilon \big(h(v,w)-h(B,v)h(B,w)\big)
\end{equation*}
and $W=-B/\varepsilon$, where $\varepsilon=1-h(B,B)$ and $B$ is the vector field metrically associated by $h$ to $\omega$, that is, $\omega[\cdot]=h(B,\cdot)$. The next step is to show that Fermat metrics provide the same family. Indeed
 it is enough to take $W=-\delta$ and
$g(\cdot,\cdot)=g_0(\cdot,\cdot)/(\beta+|\delta|_0^2)$; here $|\cdot|_0$ denotes the associated norm to $g_0$.
Then
$g(W,W)=|\delta|^2_0/(\beta+|\delta|^2_0)<1$
and
$\alpha=\beta/(\beta+|\delta|_0^2),$
so that
$g_0(\cdot,\cdot)/\beta=g(\cdot,\cdot)/\alpha.$
From the last equality it is easily concluded that a Fermat metric
is Zermelo. For the viceverse we can consider $\delta=-W$,
$\beta=(1-g(W,W))/\phi$ and
$g_0(\cdot,\cdot)=1/\phi\, g(\cdot,\cdot)$, where $\phi$ is an
arbitrary positive real function in $M$ that can be choosen in
particular constantly equal to $1$. Clearly
$g(\cdot,\cdot)/\alpha=g_0(\cdot,\cdot)/\beta,$
and a Zermelo metric is always Fermat.
\end{proof}
 We observe that Zermelo and Randers metrics are in one-to-one correspondence. The matter is different with Fermat metrics because they are not determined by the elements $g_0$, $\beta$ and $\delta$. More exactly a Fermat metric only depends on $g_0/\beta$ and $\delta$, so that $\beta$ is defined up to a positive function $\phi:M\rightarrow \R$. This degree of freedom is associated with the fact that the support of lightlike geodesics remains unchanged by conformal changes in the spacetime. From the above discussion it is easily deduced that we can associate to every Randers metric as in \eqref{randers} the standard stationary spacetime determined by $\beta=(1-h(B,B))/\phi$, $\delta=B/(1-h(B,B))$ and
\[g_0(v,w)=\varepsilon/\phi\, \big(h(v,w)-h(B,v)h(B,w)\big).
\]
\subsection{Katok examples} Zermelo metrics play a fundamental role in the classification of Randers space forms. They provide a geometric characterization of Randers metrics with constant flag curvature as it was shown in \cite{BaRoSh04}.  We observe that Zermelo metrics are also related to Katok examples. These examples (see \cite{katok73,Ziller83}) are provided by a co-Finsler metric of Randers type, that is, a Finsler metric on the cotangent bundle $H:T^*M\rightarrow \R$ given by
\begin{equation}\label{coranders}
H(x,v)=\sqrt{g^*(v,v)}+v(W),
\end{equation}
for all $(x,v)\in T^*M$, being $g^*$ the dual metric of a Riemannian metric $g$ in $M$ and $W$  a vector field in $M$ such that $g(W,W)<1$. There is a one-to-one correspondence between Finsler and co-Finsler metrics, the so-called $\mathcal L$-duality given by the Legendre transformation, so that the co-Randers metric in \eqref{coranders} determines a Finsler metric. As it was proved in \cite{HriShi96} (see also \cite{Shen01}), this metric is in fact of Randers type, more precisely, it is the Zermelo metric determined by the metric $g$ and the vector field $W$. When the vector field $W$ is Killing, the Hamiltonian properties of the geodesic flow can be used to study the closed geodesics associated to \eqref{coranders}, obtaining in this way Katok examples with a finite number of closed geodesics (see \cite{Ziller83} for an account of the geometric properties of Katok examples). Thus these examples are of Zermelo type and they can be easily expressed as Fermat metrics providing examples of spacetimes with a finite number of $t$-periodic light rays. We observe that the geometry of Katok examples has been approached directly with the Zermelo metric by C. Robles in \cite{robles07}. In this work the author obtains an explicit expression for geodesics in a class of Zermelo metrics described in the following theorem.
\begin{thm}[Robles \cite{robles07}]\label{robles} Assume that $(\mathcal{M},g)$ is equipped with an infinitesimal homothety
$W$, i.e. $\mathcal{L}_W g=\sigma g$, where $\sigma$ is constant and $\mathcal{L}$ is the Lie derivative. Let $Z$
be the Zermelo metric given by \eqref{zermelo}
defined on $M=\{x\in M : g(W,W) <1 \}\subset \mathcal{M}$. The unit
speed geodesics $\mathcal{P}:(-\epsilon, \epsilon) \onto M$ of $Z$
are given by $\mathcal{P}(t)=\phi(t,\rho(t))$, where
\begin{itemize}
\item $\rho: (-\epsilon, \epsilon)  \onto \mathcal{M}$ is a geodesic of $g$
parametrized so that $g(\fis{\rho}(t),\fis{\rho}(t))= e^{-\sigma
t}$;
\item shrinking $\epsilon$ if necessary, $\phi:(-\epsilon, \epsilon) \times U \onto M$
is the flow of $W$ defined on a neighborhood $U$ of $\rho(0)$ so
that $\rho(t)\in U$, for all $t\in (-\epsilon, \epsilon)$.
\end{itemize}
\end{thm}
This result can be easily adapted to the stationary context.
\begin{cor}\label{fermatgeodesics} Assume that $(M\times\R,l)$ is a standard stationary spacetime as in \eqref{stationarymetric} and
\[\mathcal L_{\delta}g_0=\left(\frac{ \delta(\beta+|\delta|_0^2)}{\beta+|\delta|_0^2}-\sigma\right)g_0\]
for a certain $\sigma\in \R$.  The spacelike component of light rays is given up to reparametrizations as   $x(t)=\phi(t,\rho(t))$, where
\begin{itemize}
\item $\rho: (-\epsilon, \epsilon)  \onto M$ is a geodesic of $1/(\beta+|\delta|_0^2)g_0$
parametrized so that \[g_0(\fis{\rho}(t),\fis{\rho}(t))=(\beta(\rho(t))+|\delta(\rho(t))|_0^2) e^{-\sigma
t};\]
\item shrinking $\epsilon$ if necessary, $\phi:(-\epsilon, \epsilon) \times U \onto M$
is the flow of $-\delta$ defined on a neighborhood $U$ of $\rho(0)$ so
that $\rho(t)\in U$, for all $t\in (-\epsilon, \epsilon)$.
\end{itemize}
\end{cor}
Let us describe now Katok examples. We consider a compact Riemannian manifold $(M,g)$ whose all geodesics close and they have the same minimal period that we assume equal to $2\pi$ and a closed  one-parameter subgroup of isometries  with the same minimal period $2\pi$. Then consider the Zermelo metric $Z_\alpha$ constructed with $g$ as above and $W=\alpha V$, being $\alpha\in\R$ arbitrary and $V$ the Killing vector field associated to the one-parameter subgroup that we have fixed. Clearly, if $\alpha$ is small enough, $g(W,W)<1$. The key result is that if $\alpha$ is irrational then the only closed geodesics of the Zermelo metric $Z_\alpha$ are reparametrizations of the geodesics of $(M,g)$ invariant by the one-parameter subgroup, which are a finite number in many cases. We will not describe here how to find the one-parameter subgroups and to count the number of closed geodesics, we remit the interested reader to the discussion in \cite{Ziller83}.
\begin{prop}
Let $(M,g_0)$ be one of the Riemannian manifolds $S^{2n}$, $S^{2n-1}$, $P^nC$, $P^nH$ and $P^2Ca$ endowed with the standard metrics. Then there exists a Killing field $V$ such that if we take $\delta=-\alpha V$ and $\beta=1-\alpha^2 g_0(V,V)$, with $\alpha$ a small enough irrational number, then the standard stationary spacetime \eqref{stationarymetric} has exactly $2n$, $2n$, $n(n+1)$, $2n(n+1)$ and $24$ $t$-periodic lightlike geodesics respectively.
\end{prop}
\begin{proof}
Apply the results in \cite[page 139]{Ziller83} to obtain Finsler metrics with a finite number of closed geodesics. We know because of the $\mathcal L$-duality that these metrics are Zermelo metrics (as we have shown at the beginning of this subsection), so that they can be easily expressed as  Fermat metrics of  stationary spacetimes as in Proposition \ref{prop:RZF}.
\end{proof}
\subsection{A bound for the period $T$ and reversibility of a Fermat metric.}\label{reversibility} In the paper \cite{sanchezproc99} it was shown that given a manifold $M$ whose homotopy group is non trivial, and given a non trivial homotopy class $\mathcal C$, there exists $T_{\mathcal C}>0$ such that if $T<T_{\mathcal C}$ there is no $T$-periodic causal geodesic in a standard stationary spacetime $(M\times\R,l)$ with projection in $M$ belonging to $\mathcal C$, if $T=T_{\mathcal C}$ there exists a $T$-periodic light ray, and if $T>T_{\mathcal C}$ there exists a $T$-periodic timelike geodesic. When the homotopy class is trivial the causal methods in \cite{sanchezproc99} do not apply, but we can use Finsler geometry to show that $T$ can not be arbitrarily small. We recall that the period $T$ coincides with the length of the Fermat geodesic. As the injectivity radius is continuous (see \cite[Proposition 8.4.1]{bcs}), the length of a closed geodesic must be bounded from below in compact manifolds and so the period $T$ of $t$-periodic light rays in the corresponding stationary spacetimes. In the paper \cite{rade04}, H-B. Rademacher finds a bound for the length of a geodesic loop when the curvature is $\eta$-pinched. The $\eta$ depends on the reversibility of the Finsler metric, so that in order to apply Rademacher's result to Fermat metrics we proceed to compute its reversibility.

For a compact non-reversible Finsler manifold $(M,F)$ we define
the \emph{reversibility} $\lambda=\max \{F(-X): X\in TM\, \text{\rm and } F(X)=1 \}$. Let $F$ be the Fermat metric given in \eqref{fermatmetric} and
set $p\in M$. We determine the critical points of the function
\[
f: T_p M \rightarrow \R, \ f(v)=F(-v)=1-2g_1(\delta(p),v),
\]
on $F(v)=1$,  where $g_1(\cdot, \cdot)=1/\beta\,
g_0(\cdot, \cdot)$.

Assume firstly $\delta(p)\neq 0$. It is easy to see that $\nabla
f= -2\delta$ and
\[ \nabla F= \delta + \frac{ g_1(\delta,v)\,\delta +
v}{\sqrt{g_1(\delta,v)^2+ g_1(v,v)}}
\]
with respect to $g_1$. By applying Lagrange multipliers method we deduce that $v$ is a critical point of $f$ on
$F(v)=1$ if and only if $v=\mu\, \delta$ and $F(v)=1$. Hence
\[
|\mu|\sqrt{|\delta|^2 + |\delta|^4} + \mu|\delta|^2=1,
\]
where $|\cdot |$ is the norm associated to $g_1$. Therefore we get two values of
$\mu$: one positive and the other one negative. Since
$f(v)=1-2g_1(v,\delta)$, the maximum corresponds to $\mu$
negative, that is
\[
\mu=\frac{1}{| \delta(p) |^2 - \sqrt{|\delta(p)|^2+
|\delta(p)|^4}},
\]
hence
\[
f(\mu\, \delta)= \frac{|\delta| +
\sqrt{1+|\delta|^2}}{-|\delta|+\sqrt{1+|\delta|^2}}.
\]
Now, using that $g_1= g_0/\beta $, we obtain
\[\max_{ \{v\in T_p M:\ F(v)=1     \}} f(v)=\frac{\Lambda(p) + \sqrt{1+ \Lambda^2(p)}}{-\Lambda(p)
+ \sqrt{1+\Lambda^2(p)}},\]
where $\Lambda(p)=|\delta|_0/\sqrt{\beta}$
and the equality also holds when $\delta(p)=0$, so that
\[
\lambda=\max_{x\in M} \frac{\Lambda(x) + \sqrt{1+
\Lambda^2(x)}}{-\Lambda(x) + \sqrt{1+\Lambda^2(x)}}.
\]
Moreover, as the real  function
\[f(x)=\frac{x+\sqrt{1+x^2}}{-x+\sqrt{1+x^2}}
\]
is increasing we finally conclude that
\[\lambda=\frac{\varphi+\sqrt{1+\varphi^2}}{-\varphi+\sqrt{1+\varphi^2}},
\]
where $\varphi=\max_{p\in M}|\delta|_0/\sqrt{\beta}$. Now we can translate the Rademacher's result to the Fermat language. The pinched coefficient $\eta$ in Rademacher's theorem is $\lambda/(1+\lambda)$ and the estimate on the length of geodesics loops is $\pi(\lambda+1)/\lambda$, so that putting all this in function of $\varphi$, we obtain the following result for spacetimes.
\begin{thm}\label{boundT}
Let $(M\times \R,l)$ be a stationary spacetime as in \eqref{stationarymetric} such that its Fermat metric (see \eqref{fermatmetric}) has flag curvature $K$ satisfying
 \[ \frac{\varphi+\sqrt{1+\varphi^2}}{2\sqrt{1+\varphi^2}}<K(p)<1\]for every $p\in M$. Then the period of a $t$-periodic light rays is at least \[\frac{2\pi\sqrt{1+\varphi^2}}{\varphi+\sqrt{1+\varphi^2}},\]
 where $\varphi=\max_{x\in M}|\delta|_0/\sqrt{\beta}$.
\end{thm}

\section{Conclusions and further developments}
In this work we have shown how Finsler geometry is powerful to study $t$-periodic lightlike geodesics in standard stationary spacetimes. As a matter of fact, Fermat metrics coincide with Randers and Zermelo metrics that are the families between the Finslerian metrics atracting more attention in the last years. The geometry of these metrics is still to understand and every progress in the field can be used to improve the results in this work. It will be particularly interesting to obtain expressions of the flag curvature of Zermelo metrics as a function of the curvature of the Riemannian background and the derivatives of the ``wind'' $W$. Having this in hand, we could express the hypothesis of Theorem \ref{boundT} directly in terms of $g_0$, $\beta$ and $\delta$ as it would be desirable.

\end{document}